\documentclass[11pt,a4paper]{article}
\usepackage{multirow}
\usepackage[leqno]{amsmath}
\usepackage{graphicx,latexsym,euscript,makeidx,color,bm}
\usepackage{epstopdf,subfigure}
\usepackage{amsmath,amsfonts,amssymb,amsthm,thmtools,mathtools,mathrsfs,enumerate}
\usepackage[colorlinks,linkcolor=blue,citecolor=red]{hyperref}

\usepackage{geometry}
\allowdisplaybreaks[4]
\def\dbE{\mathbb{E}}

  \def\cH{{\cal H}} \def\BH{{\bm H}}

\def\ms{\medskip}              
\def\bs{\bigskip}

\def\no{\noindent}          
         
\def\nn{\nonumber}         
\def\rf{\eqref}            
             
\def\deq{\triangleq}     \def\({\Big (}       \def\ba{\begin{aligned}}
\def\les{\leqslant}      \def\){\Big )}       \def\ea{\end{aligned}}
\def\ges{\geqslant}      \def\[{\Big[}        \def\bel{\begin{equation}\label}
\def\ti{\tilde}          \def\]{\Big]}        \def\ee{\end{equation}}
      \def\q{\quad}        
\def\h{\widehat}         \def\qq{\qquad}      
                                              
\def\a{\alpha}     \def\g{\gamma}     
   \def\D{\Delta}   \def\d{\delta}        
   \def\Th{\Theta}  \def\th{\theta}    \def\si{\sigma}
\def\f{\varphi}   \def\l{\lambda}  \def\m{\mu}      \def\e{\varepsilon}
\def\t{\tau}    \def\i{\infty}   \def\k{\kappa}   

\newtheoremstyle{thry}
{}      
{}      
{\sl}   
{}      
{\bf}   
{.}     
{.5em}  
{}      
\theoremstyle{thry}

\newtheorem{theorem}{Theorem}[section]
\newtheorem{proposition}[theorem]{Proposition}
\newtheorem{corollary}[theorem]{Corollary}

\theoremstyle{definition}
\newtheorem{definition}[theorem]{Definition}

\theoremstyle{remark}
\newtheorem{remark}[theorem]{Remark}

\def\punct{}
\newtheoremstyle{dotless}{}{}{\rm}{}{\bf}{\punct}{.5em}{}
\theoremstyle{dotless}

\makeatletter
   
   \@addtoreset{equation}{section}
   \newcommand{\setword}[2]{%
   \phantomsection
   #1\def\@currentlabel{\unexpanded{#1}}\label{#2}%
   }
\makeatother

\begin{document}

\title{\bf  Dynamic  Coalition Portfolio Selection with Recursive Utility}

\author{Hanxiao Wang
\thanks{School of Mathematical Sciences, Shenzhen University, Shenzhen
518060, China (Email: {\tt hxwang@szu.edu.cn}).}
~~and~~Chao Zhou\thanks{ Department of Mathematics and Risk Management Institute, National University of Singapore,
Singapore 119076, Singapore (Email: {\tt matzc@nus.edu.sg}).}
}

\maketitle

\no\bf Abstract.
\rm In this paper, we consider a dynamic coalition portfolio selection problem,
with each agent's objective given by an Epstein--Zin recursive utility.
To find a Pareto optimum, the coalition's problem is formulated as an optimization problem evolved by a multi-dimensional
forward-backward SDE. Since the evolution system has a forward-backward structure,
the problem is  intrinsically time-inconsistent.
With the dynamic-game point of view,
we rigorously develop an approach to finding the equilibrium Pareto investment-consumption strategy.
We find that the relationship between risk aversion and EIS has more
influence on the coalition's problem than that on the one-agent problem.
More interestingly, we show that the equilibrium Pareto consumption strategy associated
with the recursive utility is much more effective
than that associated with the CRRA expected utility,
which highlights the feature of  recursive utility that the
marginal benefit of consumption can depend on the future  consumption.

\ms

\no\bf Keywords. \rm Portfolio selection,  recursive utility, coalition optimization, cooperation game, time-inconsistency,
forward-backward SDE, equilibrium Pareto strategy.

\ms

\no\bf AMS subject classifications. \rm 91B10, 91B51, 91B70, 93E20, 49L12.

\section{Introduction}

Since the pioneering work  by von Neumann and Morgenstern \cite{von Neumann1944},
coalitions have occupied the center stage for research within game theory.
In the social sciences, a minimal requirement on allocative efficiency  is the so-called {\it Pareto optimality}
(see Aziz,  Brandt, and  Harrenstein \cite{Aziz-Brandt-Harrenstein2013}, for example).
A well-known approach of finding Pareto optima is to solve a new optimization problem with the subjective
given by some convex combination of each agent's subjective.
Then the optimum (if it exists) of the new  problem is exactly a Pareto optimum of the coalition.

\ms
On the other hand,
it is well recognized that the recursive utility developed by Epstein and Zin \cite{Epstein-Zin1989}
and Duffie and Epstein  \cite{Duffie-Epstein1992}
is a milestone in  macroeconomics and quantitative finance,
because it provides a simple but very useful tool to
separate investor's risk aversion and elasticity of intertemporal substitution (EIS, for short).
From an economic theory view of point, this is much more reasonable than the standard constant relative risk aversion
(CRRA, for short) expected utility. Since then, the recursive utility  has been widely applied on the study of asset pricing and
portfolio selection problems.
Another important feature of recursive utility is that  the marginal
benefit of consumption  depends not only on current
consumption but also (through the value function) on the trajectory of future consumption.
Such a property  is playing a  crucial role in the  fast-growing green portfolios,
because one has to consider some sustainable development issues when modeling the  performances
of environmental, social, and governance (ESG) investing;
see Pindyck and Wang  \cite{Pindyck2013} and Hong, Wang, and Yang \cite{Hong-Wang2021}, for example.

\ms

Motivated by the above, we study a coalition portfolio selection problem,
with each agent's performance described by an Epstein--Zin recursive  utility $Y_i$.
We  use  $\g$, $\a$, and $\rho_i$ to describe each agent's risk aversion, EIS, and  discount factor, respectively.
The main feature of our model is that the discount factor of each agent is allowed to be different from each each.
Indeed, since the discount rate is not only due to the interest rate,
it is more like a subjective preference for the time cost,
and then each agent should have an individual discount rate (see Dumas, Uppal, and Wang \cite{Dumas2000}
and Borovi\v{c}ka \cite{Borovicka2020}). For example  (see G\^{a}rleanu and Panageas \cite{Garleanu2015}), the old and the young in a family usually have different preferences for the time risk (or the time cost), due to which forcing them to use the same discount rate is really unreasonable.

\ms
The purpose of this paper is to study  the  coalition's  investment-consumption performance when
the model involves heterogeneous  discount rates.
Some interesting phenomenons produced by the  heterogeneous discount rates and the recursive utility
are revealed.

\ms

The main contribution of this paper can be briefly summarized as follows.

\ms

(i)  It is well-known (see EI Karoui, Peng, and Quenez \cite{El Karoui-Peng-Quenez1997})  by now  that
the recursive  utility is exactly the solution of some (nonlinear) backward stochastic
differential equation (BSDE, for short).  Combining this with the wealth equation,
the coalition's problem (denoted by Problem (C)) is reformulated as a maximizing problem with the controlled dynamics
described by a system of a forward SDE and a multi-dimensional BSDE. We find that Problem (C) is intrinsically time-inconsistent,
due to the appearance of backward state equations.
With the dynamic-game point of view introduced by Strotz \cite{Strotz1955},
we rigorously develop an approach to constructing the equilibrium  Pareto investment-consumption strategy.

\ms

(ii) We show that by applying the recursive utility, the effectiveness of equilibrium Pareto consumption strategy (EPCS, for short)
associated with CRRA utility can be significantly improved  in the following sense:

\ms

\begin{center}
\begin{tabular}{|c|c|c|}
  \hline
  Consumption strategy & Time-consistency & Heterogeneity \\
  \hline
  Pre-committed,\,\, CRRA utility & No & \textbf{Yes} \\
  \hline
  Equilibrium,\, \, CRRA utility & \textbf{Yes} & No \\
  \hline
 \textbf{Equilibrium,\,\, recursive utility} & \textbf{Yes} & \textbf{Yes} \\
  \hline
\end{tabular}

\ms
Table 1: Comparison between recursive utility and CRRA utility
\end{center}

\noindent
In the CRRA utility case, the EPCS of  agent $i$ takes the following form:
$$
 \bar c_i(t)x=\bigg(\sum_{i=1}^N\th_i(t)\bigg)^{1\over\a-1} x,\qq i=1,...,N,
$$
where $\th_i$ can be uniquely determined by solving some ordinary differential equations (ODEs, for short).
Note that the above EPCS $\bar c_i$  does not depend on the parameter $i$.
Since the agents' discount rates are heterogeneous,
such a consumption strategy is unrealistic from an economic point of view.
Thus, it is not very effective to use the CRRA utility to model a dynamic coalition portfolio selection problem.
However, this inadequacy can be improved by replacing the CRRA utility by a recursive one.
Indeed, in the recursive case, the EPCS reads:
$$
 \bar c_i(t)x=\bigg(\sum_{i=1}^N\th_i(t)\bigg)^{1\over\a-1}\th_i(t)^{1-\g-\a\over(1-\a) (1-\g)} x,\qq i=1,...,N.
$$
It is noteworthy to point out that this phenomenon is 
similar to the equilibrium investment performance in dynamic mean-variance  portfolio choices,
in which the equilibrium investment strategy is independent of current wealth for the standard criterion
(see Basak and Chabakauri \cite{Basak-Chabakauri2010}),
but is proportional to current wealth for the log return (see Dai, Jin, Kou, and  Xu \cite{Dai2021}).

\ms
(iii) We find that the relationship between risk aversion $\g$ and EIS $(1-\a)^{-1}$ has more influence
on the coalition's problem than that on the one-agent problem.

\ms

Let $c^*_i$ be the optimal consumption strategy associated with the discount rate $\rho_i$ in the one-agent problem,
and  $\bar c_i$ be the equilibrium consumption strategy associated with the discount rate $\rho_i$ in the  coalition's problem.
Let $0<\rho_i<\rho_j$, then we have

\ms

\begin{center}
\begin{tabular}{|c|c|c|}
  \hline
  Consumption strategy with $\g\in(1-\a,1)$ & Relationship between $c_i$ and $c_j$ \\
  \hline
  One-agent problem & $c^*_i\leq c^*_j$  \\
  \hline
  Coalition's problem & $\bar c_i\leq \bar c_j$\\
  \hline
\end{tabular}

\ms
Table 2: Comparison between one-agent and coalition's problems with $\g\in(1-\a,1)$
\end{center}

\ms

\begin{center}
\begin{tabular}{|c|c|c|}
  \hline
  Consumption strategy with $\g\in(0,1-\a)$ & Relationship between $c_i$ and $c_j$ \\
  \hline
  One-agent problem & $c^*_i\leq c^*_j$  \\
  \hline
  Coalition's problem & $\bar c_i\geq\bar c_j$\\
  \hline
\end{tabular}

\ms
Table 3: Comparison between one-agent and coalition's problems with $\g\in(0,1-\a)$
\end{center}
This shows that the coalition's problem
is more sensitive with the  relationship between  risk aversion and EIS than the one-agent problem.

\ms

(iv) We show that the coalition manager only needs to decide the whole coalition's investment strategy,
but has to make a decision  for each agent's consumption strategy.
This predicts  that {\it distributing the wealth in a coalition is much more complicated than making money},
because it is the common aim of each agent to make as much money as possible,
but how to spend money is nobody else's business but one's own.

\subsection{Literature review on recursive coalition problem}

The coalition problem (also called a corporation game, a social planner problem,  a group decision maker problem,
or a multi-agent problem) has been studied for a long history in economics and  social sciences.
Some  newest results can be found in Morellec and  Schsrhoff \cite{Morellec}, Aziz, Brandt, and  Harrenstein \cite{Aziz-Brandt-Harrenstein2013},
 Jackson and Yariv \cite{Jackson-Yariv2015},  Garlappi, Giammarino, and  Lazrak \cite{Garlappi2017}, and the references cited therein.
After the pioneering works \cite{Duffie-Epstein1992,Duffie-Epstein1992-1},
the coalition problem with recursive utility has been studied by
Dana and LeVan \cite{Dana1990}, Ma \cite{Ma1993}, Duffie, Geoffard, and Skiadas \cite{Duffie-Geoffard-Skiadas-1994}, Kan \cite{Kan1994},
 Dumas,  Uppal, and  Wang \cite{Dumas2000},
Anderson \cite{Anderson2005}, G\^{a}rleanu and Panageas \cite{Garleanu2015}, and  Borovi\v{c}ka \cite{Borovicka2020}, to mention a few.
However, the solutions obtained in these  works  are  pre-committed, and  we are going to find a sophisticated solution in the paper,
since the coalition's problem is  intrinsically time-inconsistent (see  Jackson and Yariv \cite{Jackson-Yariv2015}).

\ms

More importantly, most of the  above works focus on generalizing the CRRA model to the recursive
case, but the consequence of  the separation advantage between risk aversion and EIS in recursive utility is not very clear.
We will clarify this crucial issue and {\it provides a strong evidence to support the  effectiveness of introducing recursive utility in the paper.}

\subsection{Literature review on optimal control theory for FBSDEs}
In the control theory, the optimal control problem for foward-backward SDEs (FBSDEs, for short) was
initially studied by Peng  \cite{Peng1993}, and has been developed by
 Dokuchaev and Zhou \cite{Dokuchaev1999}, Lim and Zhou \cite{Lim-Zhou2001}, Yong \cite{Yong2010},
 Cvitani\'{c} and Zhang \cite{Cvitanic-Zhang2012},
Wang, Wu, and Xiong \cite{Wang-Wu-Xiong2013},
 Hu, Ji, and Xue \cite{Hu-Ji-Xue2018},
 and Sun, Wang, and Wen \cite{Sun-Wang-Wen2023}, etc.
All these results are essentially concerned with the Pontryagin's maximum principle, using a variation method.
In the paper, we focus on the time-consistency, which corresponds to  Bellman's dynamical programming principle approach.

\subsection{Literature review on time-inconsistent control theory}
The seminal paper \cite{Strotz1955} by Strotz   was the first to formalize the consistent planning problem
with  a dynamic-game point of view. Since then, the equilibrium approach by sophisticated agents has received very strong attention;
see, for example, Pollak \cite{Pollak1968}, Laibson \cite{Laibson1997}, Harris and Laibson \cite{Harris},
Basak and Chabakauri \cite{Basak-Chabakauri2010},  Cao and Werning \cite{Cao2018}, and Dai, Jin, Kou, and Xu \cite{Dai2021} on the economic side,
and Ekeland and Pirvu \cite{Ekeland2008}, Ekeland and Lazrak \cite{Ekeland2010},   Hu, Jin, and Zhou \cite{Hu-Jin-Zhou2012}, Yong \cite{Yong2012,Yong2014}, Bj\"{o}rk,  Murgoci, and Zhou \cite{Bjork-Murgoci-Zhou2014},
Bj\"{o}rk, Khapko, and Murgoci \cite{Bjork-Khapko-Murgoci2017} on the mathematical side. However,
these existing approaches cannot be applied in our coalition model with recursive utility,
because the coalition's dynamic contains a multi-dimensional backward SDE \rf{M-BSDE}.
Roughly speaking, the time-inconsistency of our model is caused purely by the backward state,
rather than by the non-exponential discount \cite{Strotz1955,Pollak1968,Ekeland2008,Ekeland2010,Yong2012,Yong2014}
or by the conditional expectation operator \cite{Basak-Chabakauri2010,Hu-Jin-Zhou2012,Bjork-Murgoci-Zhou2014,Bjork-Khapko-Murgoci2017}.
We will overcome this mathematical difficulty by modifying the approach developed in our recent work \cite{Wang-Yong-Zhou2022}.

\ms

The rest of this paper is organized as follows.
In Section \ref{sec:model}, we formulate the problem. The time-inconsistency of the model is shown in Section \ref{sec:TI}.
The main results are presented in Section \ref{sec:main-results}.
In Section \ref{sec:numercis}, we provide some numerical results and make some comparisons.
The proofs are collected in Section \ref{sec:Proofs}.

\section{The model and problem formulation}\label{sec:model}

Consider the following forward SDE for the coalition's wealth process $X$ driven by a one-dimensional Brownian motion $W$:
\bel{state}\left\{\begin{aligned}
dX(s)&=\bigg[\nu X(s)+(\m-\nu)\sum_{i=1}^N \pi_i(s)X(s)-\sum_{i=1}^N c_i(s)X(s)\bigg]ds\\
&\q+\si \sum_{i=1}^N \pi_i(s)X(s) dW(s),\qq s\in[t,T],\\
X(t)&=x,\end{aligned}\right.
\ee
where $\nu$ is the riskless interest rate, $\mu>\nu$ is the appreciation rate of one stock,
$\si>0$ is the volatility, $\pi_i\in L_+^\i(0,T)$ is agent $i$'s proportion of dollar amount
invested in the stock, and   $c_i\in L_+^\i(0,T)$ is the consumption of agent $i$.
Here, the space $L_+^\i(0,T)$ is defined by
$$
L_+^\i(0,T)=\big\{\f:[0,T]\to[0,+\i)\,|\, \f \hbox{ is essentially bounded}\big\}.
$$
For simplicity, we use the following  notations of the investment-consumption strategies from now on:
$$\pi=(\pi_1,...,\pi_N), \qq c=(c_1,...,c_N).$$

\ms
Naturally, agent $i$ wants to maximize his/her utility functional
\bel{utility}J_i(t,x;\pi_i,c_i)= Y_i(t),\ee
where $Y_i$, called an Epstein--Zin recursive utility
(see \cite{Duffie-Epstein1992,Duffie-Epstein1992-1,Duffie-Lions1992,El Karoui-Peng-Quenez1997}, for example),
is determined by
\bel{M-BSDE}Y_i(s)=\dbE_s\bigg[\int_s^T g_i\big(c_i(r)X(r),Y_i(r)\big)dr+h_i(X(T))\bigg],\qq s\in[t,T],\ee
with
$$
g_i(q,y)=\a^{-1}((1-\g)y)^{1-{\a\over 1-\g}}\[q^\a-\rho_i((1-\g)y)^{{\a\over 1-\g}}\],
$$
and
$$
 h_i(x)={x^{1-\g}\over 1-\g}.
$$
The parameter  $\g\in(0,1)$ controls the
risk aversion of the agents,  $ (1-\a)^{-1}>0$ with $\a\in(0,1)$ gives the agents' EIS,
and $\rho_i$ is the subjective discount factor of agent $i$ (which can be different for the different $i$).
If $\a=1-\g$, we have the additively separable aggregator:
\bel{ASAU}
g_i(q,y)=\a^{-1}\[q^\a-\rho_i(1-\g)y\],
\ee
and then
\bel{ASAU-CRRA}
Y_i(s)=\a^{-1}\dbE_s\bigg[\int_s^T e^{-\a\rho_i(r-s)}\big(c_i(r)X(r)\big)^\a dr+e^{-\a\rho_i(T-s)} X(T)^\a\bigg],\qq s\in[t,T],
\ee
which is exactly the standard CRRA expected utility.
Unlike the CRRA utility \rf{ASAU-CRRA}, the recursive utility as defined by \rf{M-BSDE} disentangles risk
aversion from the EIS.

\ms
Since our model is formulated in a Markovian framework,
we only focus on the following Markovian solution to \rf{M-BSDE} for simplicity.

\begin{proposition}\label{prop:well-RU}
Under the assumptions presented above, the following ODE admits a unique positive solution:
\bel{main1-prop:well-RU}
\left\{
\begin{aligned}
&\dot{\th}_i(s)+(1-\g)\th_i(s)\bigg[(\mu-\nu)\sum_{j=1}^N\pi_j(s)-\sum_{j=1}^N c_j(s)+\nu-{\g\si^2\over 2}\bigg(\sum_{j=1}^N \pi_j(s)\bigg)^2-{\rho_i\over\a}\bigg]\\
&\qq
+\a^{-1}(1-\g)\th_i(s)^{1-\g-\a\over1-\g}c_i(s)^\a=0,\\
&\th_i(T)=1.
\end{aligned}\right.
\ee
Moreover,
\bel{main2-prop:well-RU}
Y_i(s)={\th_i(s)\over 1-\g}X(s)^{1-\g},\qq s\in[0,T]
\ee
is a solution to  \rf{M-BSDE}.
\end{proposition}

In our model, each agent $i$  concerns with not only  his/her own consumption $c_i$
but also the coalition's whole wealth $X$.
This is very natural in a coalition problem; for example, each family member usually 
hopes to maximize both his/her own consumption amount
and the whole wealth possessed by the family.

\ms

The minimal requirement on allocative efficiency is the so-called Pareto optimality (see Aziz, Brandt, and Harrenstein \cite{Aziz-Brandt-Harrenstein2013}), which is defined as follows:

\begin{definition}
A pair $(\pi^*,c^*)\in (L_+^\i(0,T))^N\times ( L_+^\i(0,T))^N$ is called a {\it Pareto optimal investment-consumption strategy}
if there exists no other pair $(\pi,c)\in (L_+^\i(0,T))^N\times ( L_+^\i(0,T))^N$ such that
$$
J_i(t,x;\pi_i,c_i)\ges J_i(t,x;\pi^*_i,c^*_i),\qq i=1,...,N,
$$
and at least one of the above $N$ inequalities is strict.
\end{definition}

To find a Pareto optimal investment-consumption strategy,
we introduce the following new problem: Maximize the following objective
\bel{leader}
J^\l (t,x;\pi,c)= \sum_{i=1}^N\l_i J_i(t,x;\pi_i,c_i),
\ee
subject to \rf{state} and \rf{utility},
where $\l_i\in(0,1)$ with $\sum_{i=1}^N\l_i=1$.

\ms

By the standard results in game theory (see Yong \cite[Proposition 3.2.1]{Yong2014book}, for example),
we know that the optimal investment-consumption strategy $(\pi,c)$ of the new problem
is a Pareto optimum of the coalition.
Since we focus on clarifying the performance of heterogeneous discount rates,
we take $\l_i={1\over N}$ for $i=1,2,...,N$ to reduce the additional heterogeneous terms.
In economic languages, all the agents are treated by the coalition equally.
Then the coalition problem's objective \rf{leader} can be rewritten as follows:
\bel{leader-1}
J^\l (t,x;\pi,c)= {1\over N}\sum_{i=1}^N J_i(t,x;\pi_i,c_i).
\ee

\section{Time-inconsistency of the model}\label{sec:TI}

In this section, we shall show that  the Pareto optimum of the coalition is time-inconsistent in general.

\ms

For simplicity, we consider a special case of the model with  $\a= 1-\g$,
by which the recursive utility reduces to a CRRA one.
It follows that
\begin{align}
J^\l (t,x;\pi,c)&={1\over \a N} \dbE_t\bigg[\bigg(\sum_{i=1}^N e^{-\a\rho_i(T-t)}\bigg)X(T)^\a\nn\\
&\q+\int_t^T\bigg(\sum_{i=1}^N\ e^{-\a\rho_i(r-t)}\big(c_i(r)X(r)\big)^\a \bigg) dr\bigg].
\label{cost-N=2}
\end{align}
For any fixed initial time $t\in[0,T)$,
we introduce the following  dynamic problem over $[t,T]$: For any given $\t\in[t,T)$,
maximize
\begin{align*}
J^\l (\t,x;\pi,c)&={1\over \a N} \dbE_\t\bigg[\bigg(\sum_{i=1}^N e^{-\a\rho_i(T-t)}\bigg)X(T)^\a\\
&\q+\int_\t^T\bigg(\sum_{i=1}^N\ e^{-\a\rho_i(r-t)}\big(c_i(r)X(r)\big)^\a \bigg) dr\bigg],
\end{align*}
subject to
\begin{equation*}\left\{\begin{aligned}
dX(s)&=\bigg[\nu X(s)+(\m-\nu)\sum_{i=1}^N \pi_i(s)X(s)-\sum_{i=1}^N c_i(s)X(s)\bigg]ds\\
&\q+\si \sum_{i=1}^N \pi_i(s)X(s) dW(s),\qq s\in[\t,T],\\
X(\t)&=x.\end{aligned}\right.
\end{equation*}
Note that  the above is a standard  optimal control problem,
because $t$ is fixed in \rf{cost-N=2} and only plays a role as a parameter.
In the literature, such type of problems is usually called an {\it auxiliary problem} of the original problem.
By the standard DPP approach,
the optimal investment-consumption strategy $(\pi^{t,*},c^{t,*})$ with initial time $t$
can be explicitly obtained:
\begin{align}
&\sum_{i=1}^N\pi_i^{t,*}(s)={\mu-\nu\over (1-\a)\si^2},\qq s\in[t,T],\label{Pre-CRRA}\\
& c_i^{t,*}(s)=\bigg({N\th^t(s)\over  e^{-\a\rho_i(s-t)}}\bigg)^{1\over \a-1},\qq s\in[t,T],\q i=1,...,N, \label{Pre-CRRA1}
\end{align}
in which $\th^t$ is the unique positive solution to the following ODE:
$$
\left\{
\begin{aligned}
&\dot{\th}^t(s)+{\a(\mu-\nu)^2 \over 2(1-\a)\si^2}\th^t(s)
+(1-\a)N^{1\over \a-1}\bigg(\sum_{i=1}^N e^{\a\rho_i(s-t)\over \a-1}\bigg)\th^t(s)^{\a\over \a-1}=0,\\
&\th^t(T)={1\over N}\sum_{i=1}^N e^{-\a\rho_i(T-t)}.
\end{aligned}
\right.
$$
Since the relationship $\rho_i =\rho_j$ does not hold in general,
it is easily shown that there exists at least an $i_0\in\{1,...,N\}$ such that $c_{i_0}^{t,*}$ depends on $t$.
Indeed, without loss of generality, we can suppose that there exist two indexes $i_0,j_0\in\{1,...,N\}$ such that  $\rho_{i_0}\neq\rho_{j_0}$,
and $c_{j_0}^{t,*}$ does not depend on $t$. Then
$$
 c_{i_0}^{t,*}(s)=\bigg({N\th^t(s)\over  e^{\a\rho_{i_0}(s-t)}}\bigg)^{1\over \a-1}= c_{j_0}^{t,*}(s)
  e^{\a(\rho_{j_0}-\rho_{i_0})(s-t)\over 1-\a},\qq s\in[t,T].
$$
Note that $c_{j_0}^{t,*}$ does not depend on $t$ and $e^{\a(\rho_{j_0}-\rho_{i_0})(s-t)\over 1-\a}$
contains a parameter $t$, so the product $ c_{i_0}^{t,*}$ must depend on $t$.
Thus, the problem is time-inconsistent and the strategy given by \rf{Pre-CRRA} is a pre-committed optimum.

\ms

We remark that since $c_i\neq c_j$ in general, we cannot simplify the above problem
as a one-agent investment-consumption problem with a CRRA utility and a quasi-exponential discount,
as studied by Ekeland,  Mbodji, and Pirvu \cite{Ekeland2012}.
By the equilibrium  HJB equation approach
(see  Yong \cite{Yong2012,Yong2014}, for example) or the extended HJB equation approach
(see Bj\"{o}rk, Khapko, and Murgoci \cite{Bjork-Khapko-Murgoci2017}, for example),
the time-consistent equilibrium  investment-consumption strategy (which is also called a sophisticated solution) associated with \rf{state} and \rf{cost-N=2} can be obtained.
However, for the general case with $\g\neq 1-\a$,   the situation of controlled backward state equation is not
avoidable, and then the approach developed in \cite{Yong2012,Yong2014,Bjork-Khapko-Murgoci2017}
cannot be applied.

\section{The main result}\label{sec:main-results}

In finance,  time-consistency is an essential requirement of the practicable strategies.
When the problem is time-inconsistent, it is no longer clear what we mean by the word
``optimal", because a  strategy which is optimal for one choice of starting point in
time will generically not be optimal at a later point in time (see Bj\"{o}rk,  Murgoci, and Zhou \cite{Bjork-Murgoci-Zhou2014}).
Thus, instead of finding a Pareto optimum,
we are going to find an  equilibrium Pareto  strategy for the coalition,
with the dynamic-game theoretic framework introduced by Strotz \cite{Strotz1955}.

\subsection{Equilibrium Pareto strategy}
Inspired by \cite{Hu-Jin-Zhou2012,Yong2012,Bjork-Khapko-Murgoci2017}, we introduce the following definition,
which gives the Pareto optimality in a dynamic-game sense.

\begin{definition}\label{def:eps}
We call $(\bar\pi,\bar c)\in (L_+^\i(0,T))^N\times (L_+^\i(0,T))^N$
an {\it equilibrium Pareto investment-consumption strategy} (EPICS, for short)
of the coalition, if the following holds:
\bel{def-equilibrium1}
\limsup_{\e\to 0^+} {J^\l(t,\bar X(t);\bar\pi^\e,\bar c^\e)-J^\l(t,\bar X(t);\bar\pi,\bar c)\over\e}\les 0,\ee
for any $t\in[0,T)$ and $(\pi, c)\in (L_+^\i(0,T))^N \times ( L_+^\i(0,T))^N$, with
\bel{def-equilibrium2}
\bar\pi_i^\e(s)\deq\left\{\begin{aligned}
&\bar\pi_i(s),\q &s\in[t+\e,T];\\
& \pi_i(s), \q &s\in[t,t+\e),\end{aligned}\right.\qq i=1,...,N,
\ee
and
\bel{def-equilibrium3}
\bar c_i^\e(s)\deq\left\{\begin{aligned}
&\bar c_i(s),\q &s\in[t+\e,T];\\
& c_i(s), \q &s\in[t,t+\e),\end{aligned}\right. \qq i=1,...,N.
\ee
\end{definition}

\begin{remark}
We only consider the closed-loop strategy in the paper.
The intuition behind \autoref{def:eps} is similar to that in \cite{Hu-Jin-Zhou2012,Yong2012,Bjork-Khapko-Murgoci2017}.
At any given time $t$, the controller is playing a game  with all his/her incarnations in the future by maximizing his/her 
utility functional on $[t,t+\e)$, and knowing that he/she will lose the control of the system beyond $t+\e$.
This can be regarded a leader-follower game with infinitely many players.
\end{remark}

\begin{remark}
We remark that there are two games in our model.
The agents not only  play against their incarnations in the future, but also
play a cooperation game with each other.
\end{remark}

\subsection{Equilibrium HJB equation: Verification theorem and well-posedness}

From Section \ref{sec:TI}, we see that the pre-committed optimal strategy in the CRRA utility case
can be obtained by introducing an auxiliary problem. However,  extending this approach to the recursive utility case is challenge.
In other words, it is not easy to find a pre-committed optimal strategy for the coalition in recursive utility case,
due to the appearance of controlled multi-dimensional BSDEs.
Thus,  with the equilibrium strategy $(\bar\pi,\bar c)$ having been obtained on $[t+\e,T]$,
the optimization problem on $[t,t+\e)$, which is the  sophisticated problem of the original, still cannot be solved by the DPP approach. This brings some new difficulties to finding the equilibrium strategy (i.e., the sophisticated solution) in our model comparing with Yong \cite{Yong2012,Yong2014} and Bj\"{o}rk, Khapko, and Murgoci \cite{Bjork-Khapko-Murgoci2017}.

\ms

Motivated by our recent work \cite{Wang-Yong-Zhou2022} with Yong, we introduce the following {\it equilibrium HJB equation}
(or called an {\it extended HJB equation}):
$$
\left\{
\begin{aligned}
&\Th^i_t(t,x)+\Th^i_x(t,x)\bigg[\nu x+(\mu-\nu)\sum_{i=1}^N\bar \pi_i(t)x-\sum_{i=1}^N\bar c_i(t)x\bigg]\\
&\qq+{1\over 2}\Th^i_{xx}(t,x)\bigg(\si\sum_{i=1}^N\bar \pi_i(t)x\bigg)^2+\a^{-1}\big[(1-\g)\Th^i(t,x)\big]^{1-{\a\over 1-\g}}\\
&\qq\times\[(\bar c_i(t)x)^\a-\rho_i\big((1-\g)\Th^i(t,x)\big)^{{\a\over 1-\g}}\]=0,\qq i=1,...,N,\\
&V_t(t,x)+V_x(t,x)\bigg[\nu x+(\mu-\nu)\sum_{i=1}^N\bar \pi_i(t)x-\sum_{i=1}^N\bar c_i(t)x\bigg]\\
&\qq+{1\over 2}V_{xx}(t,x)\bigg(\si\sum_{i=1}^N\bar \pi_i(t)x\bigg)^2+{1\over N}\sum_{i=1}^N\a^{-1}\big[(1-\g)\Th^i(t,x)\big]^{1-{\a\over 1-\g}}\\
&\qq\times\[(\bar c_i(t)x)^\a-\rho_i\big((1-\g)\Th^i(t,x)\big)^{{\a\over 1-\g}}\]=0,\\
&\Th^i(T,x)={x^{1-\g}\over 1-\g},\qq i=1,...,N,\qq V(T,x)={x^{1-\g}\over 1-\g},
\end{aligned}\right.
$$
in which the equilibrium investment-consumption strategy $(\bar\pi,\bar c)$
(if it exists)
is  determined by the following Hamiltonian:
\begin{align}\label{Hamiltonian-condition}
&\cH(t,x,\Th(t,x),V_x(t,x),V_{xx}(t,x), \bar \pi(t),\bar c(t))\nn\\
&\q=\sup_{\pi,c}\cH(t,x,\Th(t,x),V_x(t,x),V_{xx}(t,x), \pi, c),
\end{align}
with
\begin{align}\label{Hamiltonian}
&\cH(t,x,\Th(t,x),V_x(t,x),V_{xx}(t,x), \pi, c)\nn\\
&\q=V_x(t,x)\bigg[\nu x+(\mu-\nu)\sum_{i=1}^N \pi_i x-\sum_{i=1}^N c_ix\bigg]+{1\over 2}V_{xx}(t,x)
\bigg(\si\sum_{i=1}^N \pi_ix\bigg)^2\nn\\
&\qq+{1\over N}\sum_{i=1}^N g_i(c_ix,\Th_i(t,x))\nn\\
&\q=V_x(t,x)\bigg[\nu x+(\mu-\nu)\sum_{i=1}^N \pi_ix-\sum_{i=1}^N c_ix\bigg]+{1\over 2}V_{xx}(t,x)\bigg(\si\sum_{i=1}^N \pi_ix\bigg)^2\nn\\
&\qq+{1\over N}\sum_{i=1}^N\a^{-1}\big((1-\g)\Th^i(t,x)\big)^{1-{\a\over 1-\g}}
\[(c_ix)^\a-\rho_i\big((1-\g)\Th^i(t,x)\big)^{{\a\over 1-\g}}\].
\end{align}

Clearly, if the equilibrium strategy exists, then we have
\bel{EIS}
\sum_{i=1}^N \bar \pi_i(t)x={(\nu-\mu)V_x(t,x)\over \si^2 V_{xx}(t,x)},
\ee
and
\bel{ECS}
\bar c_i(t)x=\left[NV_x(t,x)\right]^{1\over\a-1}[(1-\g)\Th^i(t,x)]^{1-\g-\a\over (1-\a)(1-\g)},\qq i=1,...,N.
\ee
Now we make  the ansatz:
\begin{align}
&\Th^i(t,x)={1\over 1-\g}x^{1-\g}\th_i(t),\qq V(t,x)={v(t)\over 1-\g} x^{1-\g}={\sum_{i=1}^N\th_i(t)\over N(1-\g)} x^{1-\g}.\label{ansatz}
\end{align}
Then
\begin{align*}
&\Th^i_x(t,x)=\th_i(t) x^{-\g},\q \Th^i_{xx}(x)=-\g \th_i(t)x^{-\g-1},\\
&V_x(t,x)={1\over N}\sum_{i=1}^N\th_i(t)x^{-\g},\qq V_{xx}(t,x)=-{\g\over N}\sum_{i=1}^N\th_i(t)x^{-\g-1}.
\end{align*}
The equilibrium investment strategy \rf{EIS} and the equilibrium consumption strategy \rf{ECS} become
\begin{align}
\sum_{i=1}^N \bar \pi_i(t)x={(\mu-\nu)\over \g\si^2}x,\label{U}
\end{align}
and
\bel{C}
 \bar c_i(t)x=\bigg(\sum_{j=1}^N\th_j(t)\bigg)^{1\over\a-1}\th_i(t)^{1-\g-\a\over (1-\a)(1-\g)} x,\qq i=1,...,N,
\ee
respectively, with
\bel{M-HJB}
\left\{
\begin{aligned}
&\dot{\th}_i(t)+(1-\g)\th_i(t)\bigg[\nu+{(\mu-\nu)^2\over 2\g\si^2}-\bigg(\sum_{j=1}^N\th_j(t)\bigg)^{1\over\a-1}
\sum_{k=1}^N\th_k(t)^{1-\g-\a\over (1-\a)(1-\g)}\bigg]\\
&\q-(1-\g)\rho_i\a^{-1}\th_i(t)
+\a^{-1}(1-\g)\th_i(t)^{1-\g-\a\over(1-\a)(1-\g)}\bigg(\sum_{j=1}^N\th_j(t)\bigg)^{\a\over\a-1}=0,\\
&\th_i(T)=1.
\end{aligned}\right.
\ee
We call \rf{M-HJB} an {\it equilibrium ODE}.

\begin{theorem}\label{thm:VT}
Suppose that the equilibrium ODE  \rf{M-HJB} has a positive solution.
Then the  strategy $(\bar\pi,\bar c)$ given by \rf{U} and \rf{C} is an EPICS of the coalition.
\end{theorem}

\ms

To guarantee the well-posedness of \rf{M-HJB}, we introduce the following conditions.

\bs

\noindent
{\bf (A1)} The coefficients $\g$, $\a$, and $\rho_i$ satisfy
\bel{Assum-HJB}
\g\in(0,1),\qq \max_{i=1,...,N}\rho_i\les\a\nu+{(\mu-\nu)^2\over 2\g\si},
\ee
or
\bel{Assum-HJB'}
\g\in[1-\a,1).
\ee

\bs

\begin{theorem}\label{thm:well-posedness}
Let {\rm (A1)} hold.
Then the equilibrium ODE \rf{M-HJB} admits a unique positive solution.
\end{theorem}

Combining \autoref{thm:VT} and \autoref{thm:well-posedness} together,
we have the following result immediately.

\begin{corollary}
Let {\rm (A1)} hold.
Then the  strategy $(\bar\pi,\bar c)$ given by \rf{U} and \rf{C} is an EPICS of the coalition.
\end{corollary}

\section{Interpretation and numerics}\label{sec:numercis}

The equilibrium investment strategy is given by
\begin{equation*}
\sum_{i=1}^N \bar \pi_i(t)={(\mu-\nu)\over \g\si^2},\qq t\in[0,T].
\end{equation*}
It is independent of the discount rate $\rho_i$, which coincides with the case of one-agent
optimal investment-consumption problem.
We remark that the coalition  only needs to decide the sum $\sum_{i=1}^N \bar\pi_i={(\mu-r)\over \g\si^2}$ of each agent's investment strategy.

\ms

The equilibrium consumption strategy is given by
\begin{equation*}
\bar c_i(t)=\bigg(\sum_{j=1}^N\th_j(t)\bigg)^{1\over\a-1}\th_i(t)^{1-\g-\a\over (1-\a)(1-\g)} ,\qq i=1,...,N.
\end{equation*}
It shows that the coalition needs to choose an explicit consumption strategy for each agent.
Thus, although the agents are playing  a cooperation game, how to share the wealth  is still much more troublesome than
making money itself.

\subsection{Monotonic property of equilibrium consumption strategy}

\begin{proposition}\label{prop:ci-cj}
Let {\rm (A1)} hold. Then for any two given discount rates $\rho_{j_0}\ges\rho_{i_0}\ges 0$,
the corresponding equilibrium consumption strategies $\bar c_{i_0}$ and $\bar c_{j_0}$ satisfy $\bar c_{i_0}\ges \bar c_{j_0}$ if $\g\in(0,1-\a]$, and  $\bar c_{i_0}\les \bar c_{j_0}$ if $\g\in[1-\a,1)$.
\end{proposition}

\begin{proof}
Recall that the agents' equilibrium consumption strategies are given by
$$
 \bar c_i(t)=\bigg(\sum_{i=1}^N\th_i(t)
\bigg)^{1\over\a-1}  \th_i(t)^{1-\g-\a\over (1-\g)(1-\a)},\qq i=1,...,N.
$$
From the proof of \autoref{thm:well-posedness}, we know that
$$
\th_{i_0}\ges \th_{j_0}>0.
$$
If $\g\in(0,1-\a]$, we have ${1-\g-\a\over (1-\g)(1-\a)}\ges 0$, and then $\bar c_{i_0}\ges \bar c_{j_0}$.
Similarly, we have $\bar c_{i_0}\les \bar c_{j_0}$ if $\g\in[1-\a,1)$.
\end{proof}

The parameter $\g\in(0,1)$ controls the risk aversion of the agents.
If the risk aversion coefficient $\g$ is smaller than the reciprocal $(1-\a)$ of the EIS $(1-\a)^{-1}$,
the agent with a smaller discount factor will be distributed more rewards
than the one with a bigger discount factor.
However, when the risk aversion coefficient $\g$ is bigger than the reciprocal  of the EIS,
the agent with a smaller discount factor will be distributed less rewards
than the one with a bigger discount factor.
Thus, the relation between the risk aversion and the EIS has an important influence on the coalition's
taste of the discount factors.

\subsection{Comparison with one-agent model}

In the one-agent model, the wealth equation and the objective are given by
$$\left\{\begin{aligned}
dX(s)&=\[\nu X(s)+(\m-\nu) \pi(s)X(s)-c(s)X(s)\]ds\\
&\q+\si \pi(s)X(s) dW(s),\qq s\in[t,T],\\
X(t)&=x,\end{aligned}\right.
$$
and
$$
J(t,\xi;\pi,c)=Y(t),
$$
respectively, with
$$Y(s)=\dbE_s\bigg[\int_s^T g(c(r)X(r),Y(r))dr+h(X(T))\bigg],\qq s\in[t,T],$$
and
$$
g(q,y)=\a^{-1}((1-\g)y)^{1-{\a\over 1-\g}}\[q^\a-\rho((1-\g)y)^{{\a\over 1-\g}}\],\qq h(x)={x^{1-\g}\over 1-\g}.
$$
By the recursive DPP established by Peng \cite{Peng1997},
the optimal investment-consumption strategy $(\pi^*,c^*)$ is given by
\bel{C-one}
\pi^*(t)={(\mu-\nu)\over \g\si^2},\qq c^*(t)=\th(t)^{-\a\over (1-\a)(1-\g)},
\ee
with $\th$ being the unique positive solution of the following ODE:
\bel{HJB-One}
\left\{
\begin{aligned}
&\dot{\th}(t)+(1-\g)\th(t)\left[\nu+{(\mu-\nu)^2\over 2\g\si^2}\right]\\
&\qq-(1-\g)\rho\a^{-1}\th(t)
+(\a^{-1}-1)(1-\g)\th(t)^{-\a\over(1-\g)(1-\a)}=0,\\
&\th(T)=1.
\end{aligned}\right.
\ee
We denote the optimal investment-consumption strategy associated with $\rho=\rho_i$ by $(\pi^*_i,c^*_i)$.
It is well-known that for the one-agent problem, we always have $c^*_i\les c^*_j$ if $\rho_i\les \rho_j$.
But for the multi-agent problem,  if $\g\in(0,1-\a)$, we can have  $\bar c_i\ges\bar c_j$ for $\rho_i\les\rho_j$.
This is a new feature of our model.

\subsection{Comparison with  CRRA model}\label{subsec:CRRA}

In this subsection, we shall compare the equilibrium strategy derived in our model with that
derived in the CRRA utility case.

\ms

\textbf{Pre-committed strategy with CRRA utility.}
When $\g=1-\a$, the recursive utility reduces to a CRRA utility.
Recall from \rf{Pre-CRRA1} that at time $t$, the pre-committed optimal consumption strategy $c^{t,*}_i$ of agent $i$ reads:
\begin{align}
 c_i^{t,*}(s)=\bigg({N\th^t(s)\over  e^{-\a\rho_i(s-t)}}\bigg)^{1\over \a-1},\qq s\in[t,T],
\label{Pre-CRRA11}
\end{align}
which depends on $i$.  Since this strategy is time-inconsistent, it is not flexible.

\ms

\textbf{Equilibrium strategy with CRRA utility.}
Taking $\g=1-\a$ in \rf{C}, we see that the  equilibrium consumption strategy of agent $i$ with CRRA utility is given by:
\bel{C1-CRRA}
 \bar c_i(t)=\bigg(\sum_{i=1}^N\th_i(t)\bigg)^{1\over\a-1} ,\qq t\in[0,T].
\ee
It is time-consistent, but it is surprising that $\bar c_i$ does not depend on $i$.
In other words, the agents take the same consumption strategy as each other,
which is unreasonable from a economic view of point.

\ms

\textbf{Equilibrium strategy with recursive utility.}
The equilibrium  consumption strategy of agent $i$ in our model is given by
\bel{}
 \bar c_i(t)=\bigg(\sum_{j=1}^N\th_j(t)\bigg)^{1\over\a-1}\th_i(t)^{1-\g-\a\over(1-\a) (1-\g)} ,\qq t\in[0,T],\qq i=1,...,N,
\ee
which is  time-consistent and also depends on $i$.
Thus, the equilibrium consumption strategy in our model appears more reasonable than that associated with the CRRA utility.

\bs
\begin{center}
\textbf{Analysis for the above phenomena}
\end{center}

From \rf{Hamiltonian-condition} and \rf{Hamiltonian}, with the ansatz \rf{ansatz}, we see that the
equilibrium investment-consumption strategy $(\bar\pi,\,\bar c)$ is determined by maximizing the following Hamiltonian:
\begin{align}\label{Hamiltonian-RW}
\BH(t,\th(t),v(t), \pi, c)
&=v(t)\bigg[\nu +(\mu-\nu)\sum_{i=1}^N \pi_i-\sum_{i=1}^N c_i\bigg]-{1\over 2}\g v(t)\bigg(\si\sum_{i=1}^N \pi_i\bigg)^2\nn\\
&\q+{1\over N}\sum_{i=1}^N\ti g_i(c_i,\th_i(t)),
\end{align}
with
\begin{align}
&\ti g_i(c_i,\th_i(t))=\a^{-1}\th_i(t)^{1-{\a\over 1-\g}}
\Big(c_i^\a-\rho_i\th_i(t)^{{\a\over 1-\g}}\Big),
\end{align}
and
\bel{}
v={1\over N}\sum_{i=1}^N\th_i.
\ee
Note that in the recursive utility case, the marginal benefit of agent $i$'s consumption $c_i$ at time $t$ is given by
\begin{align}
\partial_{c_i}\[{1\over N}\sum_{i=1}^N \ti g_i(c_i(t),\th_i(t))\]={\th_i(t)^{1-{\a\over 1-\g}}c_i(t)^{\a-1}\over N},
\end{align}
which depends not only on  agent $i$'s current consumption $c_i(t)$ but also (through the value function $\th_i(t)$) on the trajectory of
coalition's future consumptions $\{c_k\}_{k=1,...,N}$.
It leads to that $\partial_{c_i}\BH(t,\th(t),v(t), \pi, c)$, given by
\begin{align}
\partial_{c_i}\BH(t,\th(t),v(t), \pi, c)
&=-v(t)+\partial_{c_i}\[{1\over N}\sum_{i=1}^N \ti g_i(c_i(t),\th_i(t))\]\nn\\
&=-v(t)+{\th_i(t)^{1-{\a\over 1-\g}}c_i^{\a-1}\over N},
\end{align}
depends on both coalition's value function $v$ and agent $i$'s value function $\th_i$.
Thus, the corresponding equilibrium  consumption strategy $\bar c_i$ of agent $i$  depends on the parameter $i$.

\ms
In the CRRA utility case, the aggregator becomes an additively separable one:
\bel{}
\ti g_i(c_i,\th_i(t))=\a^{-1}
\Big(c_i^\a-\rho_i\th_i(t)\Big),
\ee
and then the marginal benefit of agent $i$'s consumption $c_i$ at time $t$ reduces to
\begin{align}
\partial_{c_i}\[{1\over N}\sum_{i=1}^N \ti g_i(c_i(t),\th_i(t))\]
=\partial_{c_i}\[{1\over N}\sum_{i=1}^N \a^{-1}
\big(c_i^\a-\rho_i\th_i(t)\big)\]
={c_i(t)^{\a-1}\over N},
\end{align}
which only depends on agent $i$'s current consumption $c_i(t)$. Then,
\begin{align}
\partial_{c_i}\BH(t,\th(t),v(t), \pi, c)
&=-v(t)+{c_i^{\a-1}\over N}.
\end{align}
Thus, in this case the equilibrium  consumption strategy $\bar c_i$ of agent $i$ is given by coalition's value function $v$,
which is independent of  the parameter $i$.

\ms
The more essential reason is that the time-inconsistent problem is an optimization problem
on the local time horizon $[t,t+\e]$. For the model with CRRA utility, the  utility
is affected by the discount factor $\rho_i$ only through the discounting function $e^{-\rho_i(s-t)}$. When $\e>0$ is small enough,
$e^{-\rho_i(s-t)}\approx 1$ over $[t,t+\e]$ for each $i=1,2,...,N$. Then the agents face almost the same
problem as each other over $[t,t+\e]$. 
However, for the  model with recursive utility, the discount factor $\rho_i$ affects the utility
not only by  the discounting function $e^{-\rho_i(s-t)}$ but also by the value function $\th_i$.
Thus, the agents' problems can be still different from each other, even if $\e>0$ is very small.

\subsection{Numerics}
\begin{center}
\textbf{The setting}
\end{center}
\autoref{fig:rho1}: We set $T=1$, $\mu=0.08$, $\nu=0.02$, $\si=0.15$, $\rho_1=0.01$, $\rho_2=0.2$,
$\g=0.1$, $\a=0.3$, $\l_1=0.5$, and $\l_2=0.5$. Note that in this setting,  the relation $\g\in(0,1-\a)$ holds.
In \autoref{fig:rho1} (a), $(\th_1,\th_2)$ is the solution of the equilibrium HJB equation \rf{M-HJB};
the functions $\th_3$ and $\th_4$ are the solutions of the  HJB equation \rf{HJB-One} in the one-agent model
associated with $\rho=\rho_1=0.01$ and $\rho=\rho_2=0.2$, respectively.
In \autoref{fig:rho1} (b),  $(c_1,c_2)$ are the equilibrium consumption strategies \rf{C}
of agent $1$ and agent $2$ in the two-agent model, respectively;
the functions $c_3$ and $c_4$ are the optimal consumption strategies \rf{C-one} in the one-agent model
associated with $\rho=\rho_1=0.01$ and $\rho=\rho_2=0.2$, respectively.

\bs

\noindent
\autoref{fig:rho2}: We set $T=1$, $\mu=0.2$, $\nu=0.1$, $\si=0.05$, $\rho_1=0$, $\rho_2=0.18$,
$\g=0.8$, $\a=0.25$, $\l_1=0.5$, and $\l_2=0.5$. In this setting, we have $\g\in(1-\a,1)$.
The meaning of $(\th_1,\th_2,\th_3,\th_4)$ and $(c_1,c_2,c_3,c_4)$ is similar to that in  \autoref{fig:rho1}.

\begin{figure}[h!]
\centering
\subfigure[]{ \includegraphics[width=0.35\linewidth]{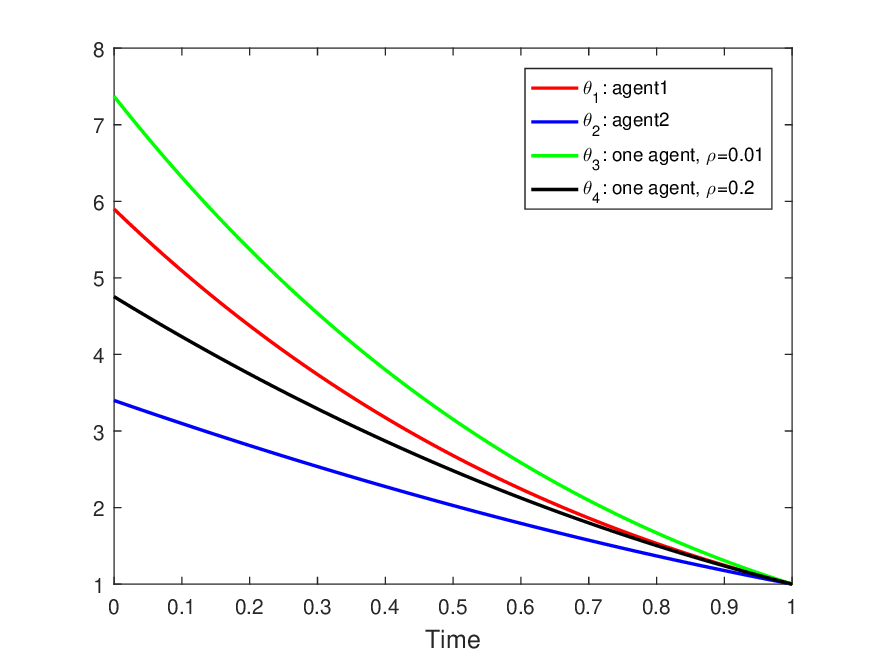}}
\subfigure[]{ \includegraphics[width=0.35\linewidth]{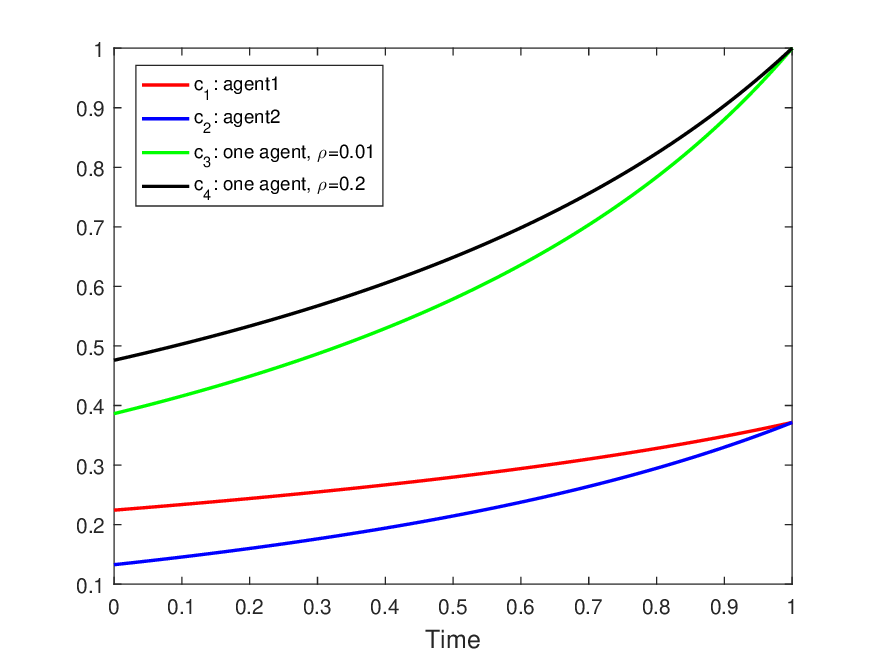}}
    \caption{The case with $\g\in(0,1-\a)$.
    }
    \label{fig:rho1}
\end{figure}

\bs
\begin{figure}[h!]
\centering
\subfigure[]{ \includegraphics[width=0.35\linewidth]{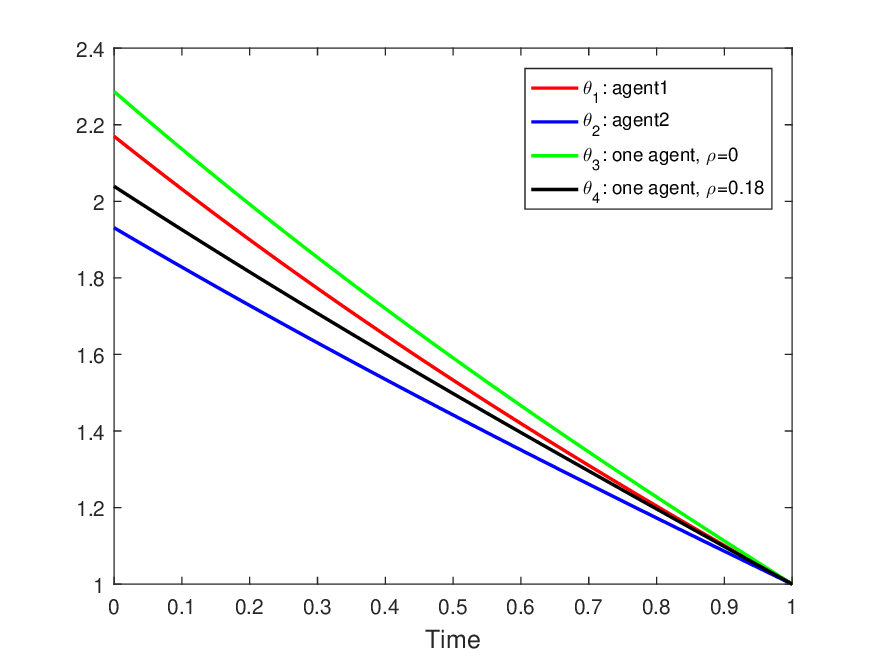}}
\subfigure[]{ \includegraphics[width=0.35\linewidth]{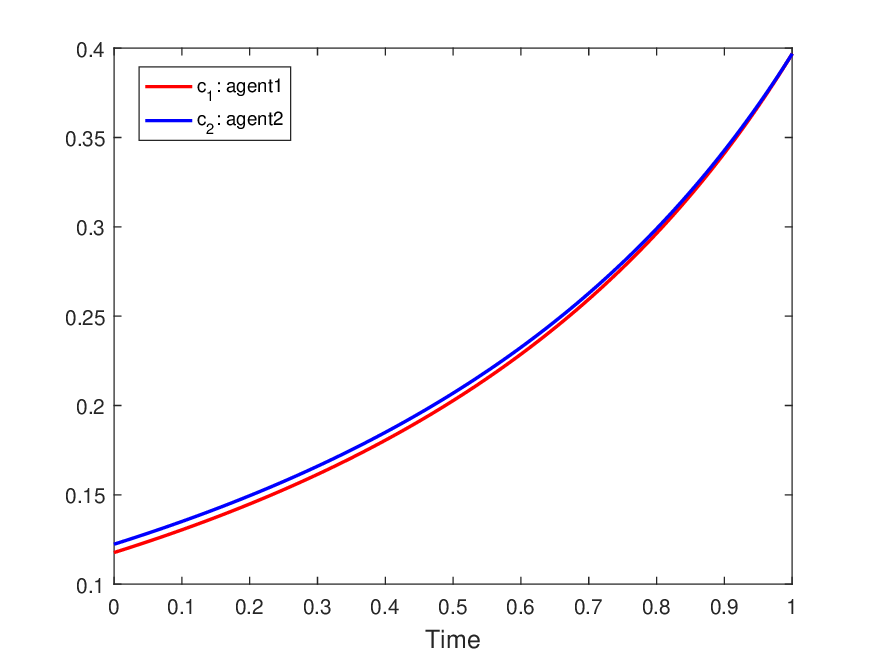}}
\subfigure[]{ \includegraphics[width=0.35\linewidth]{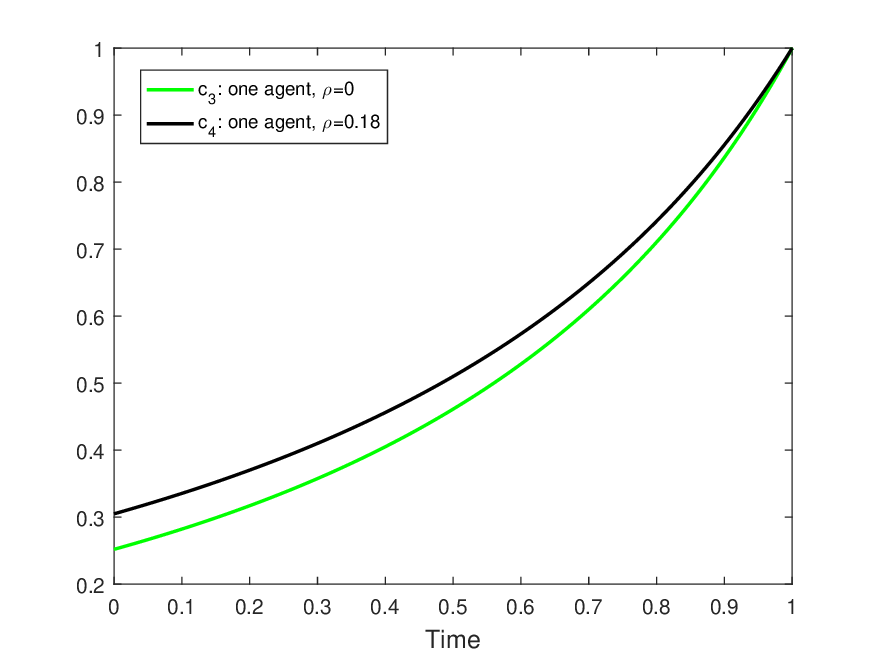}}
    \caption{The case with $\g\in(1-\a,1)$.
    }
    \label{fig:rho2}
\end{figure}

\ms
From \autoref{fig:rho1} (a) and \autoref{fig:rho2} (a), we see that $\th_1\ges \th_2$ and
 $\th_3\ges \th_4$ for both $\g\in(0,1-\a)$ and $\g\in(1-\a,1)$.
 From \autoref{fig:rho1} (b) and \autoref{fig:rho2} (c), we see that $c_3\les c_4$ holds
 for both  $\g\in(0,1-\a)$ and $\g\in(1-\a,1)$ in the one-agent model.
But in the multi-agent model, from \autoref{fig:rho2} (b) and \autoref{fig:rho2} (c) we see that
$c_1\ges c_2$ for $\g\in(0,1-\a)$ and $c_1\les c_2$ for $\g\in(1-\a,1)$.
Thus, the numerics coincides with our theoretical results.

\section{Appendix: Proof of the main results}\label{sec:Proofs}

\subsection*{Appendix A: Proof of \autoref{prop:well-RU}}
If \rf{main1-prop:well-RU} admits a unique positive solution,
then \rf{main2-prop:well-RU} can be obtained by applying the It\^{o} formula directly.
The uniqueness of positive solutions to \rf{main1-prop:well-RU} can be obtained by a standard method.
Thus, we only need to prove that \rf{main1-prop:well-RU} really admits a positive solution.
This suffices to show that if $\th_i,i=1,...,N$ is a positive solution of \rf{main1-prop:well-RU}
on $[t_0,T]$, then
$$
{\d}\les\th_i(s)\les \kappa,\q s\in[t_0,T];\q i=1,...,N,
$$
for some positive constants $\d,\kappa>0$ independent of $t_0$.
Let $\bar\th_i$ be the unique solution to the following linear ODE:
$$
\left\{
\begin{aligned}
&\dot{\bar\th}_i(s)+(1-\g)\bar\th_i(s)\bigg[(\mu-\nu)\sum_{j=1}^N\pi_j(s)-\sum_{j=1}^N c_j(s)+\nu-{\g\si^2\over 2}\bigg(\sum_{j=1}^N \pi_j(s)\bigg)^2-{\rho_i\over\a}\bigg]=0,\\
&\bar\th_i(T)=1.
\end{aligned}\right.
$$
Clearly, $\bar\th_i>0$. By the comparison theorem of ODEs, we have
\bel{proof1-prop:well-RU}
\th_i(s)\ges \d:=\inf_{r\in[0,T]}\bar\th_i(r)>0,\qq s\in[t_0,T].
\ee
Now we establish the upper bound estimate of $\th_i$.
If $\g\in(0,1-\a]$, then
$$
\a^{-1}(1-\g)\th_i(s)^{1-\g-\a\over1-\g}c_i(s)^\a\les
\a^{-1}(1-\g)[1+\th_i(s)]c_i(s)^\a.
$$
Let $\ti\th_i$ be the unique positive solution to the following linear ODE:
$$
\left\{
\begin{aligned}
&\dot{\ti\th}_i(s)+(1-\g)\ti\th_i(s)\bigg[(\mu-\nu)\sum_{j=1}^N\pi_j(s)-\sum_{j=1}^N c_j(s)+\nu-{\g\si^2\over 2}\bigg(\sum_{j=1}^N \pi_j(s)\bigg)^2-{\rho_i\over\a}\bigg]\\
&\qq
+\a^{-1}(1-\g)[1+\ti\th_i(s)]c_i(s)^\a=0,\\
&\ti \th_i(T)=1.
\end{aligned}\right.
$$
By the comparison theorem of ODEs again, we have
\bel{proof2-prop:well-RU}
\th_i(s)\les \kappa:=\sup_{r\in[0,T]}\ti\th_i(r),\qq s\in[t_0,T].
\ee
If $\g\in(1-\a,1)$, then
$$
\a^{-1}(1-\g)\th_i(s)^{1-\g-\a\over1-\g}c_i(s)^\a\les
\a^{-1}(1-\g)\d^{1-\g-\a\over1-\g}c_i(s)^\a,
$$
in which $\d>0$ is given by \rf{proof1-prop:well-RU}.
Let $\h\th_i$ be the unique positive solution to the following linear ODE:
$$
\left\{
\begin{aligned}
&\dot{\h\th}_i(s)+(1-\g)\h\th_i(s)\bigg[(\mu-\nu)\sum_{j=1}^N\pi_j(s)-
\sum_{j=1}^N c_j(s)+\nu-{\g\si^2\over 2}\bigg(\sum_{j=1}^N \pi_j(s)\bigg)^2-{\rho_i\over\a}\bigg]\\
&\qq
+\a^{-1}(1-\g)\d^{1-\g-\a\over1-\g}c_i(s)^\a=0,\\
&\h \th_i(T)=1.
\end{aligned}\right.
$$
By the comparison theorem of ODEs again, we have
\bel{proof3-prop:well-RU}
\th_i(s)\les \kappa:=\sup_{r\in[0,T]}\h\th_i(r),\qq s\in[t_0,T].
\ee
Combining \rf{proof1-prop:well-RU}, \rf{proof2-prop:well-RU}, and \rf{proof3-prop:well-RU} together, we complete the proof.

\subsection*{Appendix B: Proof of \autoref{thm:VT}}
Let $(\bar\pi_i^\e,\bar c^\e_i)$ be defined as \rf{def-equilibrium2} and \rf{def-equilibrium3}.
Then the corresponding state $X^\e$ and utility $Y_i^\e$
are given by
$$\left\{\begin{aligned}
dX^\e(s)&=\bigg[\nu X^\e(s)+(\m-\nu)\sum_{i=1}^N \bar\pi_i(s)X^\e(s)-\sum_{i=1}^N\bar c_i(s)X^\e(s)\bigg]ds\\
&\q+\si \sum_{i=1}^N \bar\pi_i(s)X^\e(s) dW(s),\qq s\in[t+\e,T];\\
dX^\e(s)&=\bigg[\nu X^\e(s)+(\m-\nu)\sum_{i=1}^N \pi_i(s)X^\e(s)-\sum_{i=1}^N c_i(s)X^\e(s)\bigg]ds\\
&\q+\si \sum_{i=1}^N \pi_i(s)X^\e(s) dW(s),\qq s\in[t,t+\e),\\
X^\e(t)&=\bar X(t),\end{aligned}\right.
$$
and
\begin{align}
Y^\e_i(s)&=\dbE_s\bigg[\int_s^T g_i(\bar c_i(r) X^\e(r),Y^\e_i(r))dr+h_i(X^\e(T))\bigg],\qq s\in[t+\e,T];\nn\\
Y^\e_i(s)&=\dbE_s\bigg[\int_s^{t+\e} g_i( c_i(r) X^\e(r),Y^\e_i(r))dr+Y^\e_i(t+\e)\bigg],\qq s\in[t,t+\e],\label{Proof-Y1}
\end{align}
respectively. 
By applying the It\^{o} formula to the mapping $s\mapsto {1\over 1-\g}X^\e(s)^{1-\g}\th_i(s)$ on $[t+\e,T]$, 
we get
$$
Y_i^\e(t+\e)={1\over 1-\g}X^\e(t+\e)^{1-\g}\th_i(t+\e),
$$
where $\th_i$ is the solution to \rf{M-HJB}.
Substituting the above into \rf{Proof-Y1} yields that
\begin{align}
Y^\e_i(s)&=\dbE_s\bigg[\int_s^{t+\e} g_i( c_i(r) X^\e(r),Y^\e_i(r))dr+{\th_i(t+\e) X^\e(t+\e)^{1-\g}\over 1-\g}\bigg],\qq s\in[t,t+\e].\nn
\end{align}
Next, by applying the It\^o formula to the mapping $s\mapsto{\th_i(s) X^\e(s)^{1-\g}\over 1-\g}$ on $[t,t+\e]$, we have
\begin{align}
Y^\e_i(s)&
=\dbE_s\bigg[\int_s^{t+\e} g_i( c_i(r) X^\e(r),Y^\e_i(r))dr+{\th_i(t)\bar X(t)^{1-\g}\over 1-\g}\nn\\
&\q+{\th_i(t+\e) X^\e(t+\e)^{1-\g}\over 1-\g}-{\th_i(t)\bar X(t)^{1-\g}\over 1-\g}\bigg]\nn\\
&=\dbE_s\bigg[\int_s^{t+\e} g_i( c_i(r) X^\e(r),Y^\e_i(r))dr+{\th_i(t)\bar X(t)^{1-\g}\over 1-\g}\nn\\
&\q+\int_s^{t+\e}X^\e(r)^{1-\g}\bigg[ {\dot{\th}_i(r) \over 1-\g}+\th_i(r)\Big[\nu +(\m-\nu)\sum_{j=1}^N \pi_j(r)-\sum_{j=1}^N c_j(r)\Big]\nn\\
&\q-{\g\over 2}\th_i(r)\si^2 \Big(\sum_{j=1}^N \pi_j(r)\Big)^2\bigg]dr\bigg].
\label{proof-VT-2}
\end{align}
By \autoref{prop:well-RU}, the following ODE admits a unique positive soluiton $\th_i^\e$:
$$
\left\{
\begin{aligned}
&\dot{\th}^\e_i(s)+(1-\g)\th^\e_i(s)\bigg[(\mu-\nu)\sum_{j=1}^N\pi_j(s)-\sum_{j=1}^N c_j(s)
+\nu-{\g\si^2\over 2}\bigg(\sum_{j=1}^N \pi_j(s)\bigg)^2-{\rho_i\over\a}\bigg]\\
&\qq
+\a^{-1}(1-\g)\th^\e_i(s)^{1-\g-\a\over1-\g}c_i(s)^\a=0,\\
&\th^\e_i(t+\e)=\th_i(t+\e).
\end{aligned}\right.
$$
Then
$$
Y_i^\e(r)={\th_i^\e (r)X^\e(r)^{1-\g}\over 1-\g},\qq r\in[t,t+\e],
$$
which implies that
$$
g_i( c_i(r) X^\e(r),Y^\e_i(r))=\a^{-1}\[c_i(s)^\a\th_i^\e(r)^{1-{\a\over1-\g}}-\rho_i\th_i^\e(r)\]X^\e(r)^{1-\g},\qq r\in[t,t+\e].
$$
Substituting the above into \rf{proof-VT-2} yields that
\begin{align}
Y^\e_i(t)&
=\dbE_s\bigg[\int_s^{t+\e} g_i( c_i(r) X^\e(r),Y^\e_i(r))dr+{\th_i(t)\bar X(t)^{1-\g}\over 1-\g}\nn\\
&\q+\int_s^{t+\e}\bigg[ {\dot{\th}_i(r) X^\e(r)^{1-\g}\over 1-\g}+\th_i(r)X^\e(r)^{-\g}\Big(\nu X^\e(r)+(\m-\nu)\sum_{i=1}^N \pi_i(r)X^\e(r)\nn\\
&\q-\sum_{i=1}^N c_i(r)X^\e(r)\Big)-{\g\over 2}\th_i(r)X^\e(r)^{-\g-1}\si^2 \Big(\sum_{i=1}^N \pi_i(r)\Big)^2X^\e(r)^2\bigg]dr\bigg]\nn\\
&=\dbE_t\bigg[{\th_i(t)\bar X(t)^{1-\g}\over 1-\g}+\int_t^{t+\e}X^\e(r)^{1-\g}\bigg[ {\dot{\th}_i(r) \over 1-\g}\nn\\
&\q+\th_i(r)\Big(\nu +(\m-\nu)\sum_{j=1}^N \pi_j(r)-\sum_{j=1}^N c_j(r)\Big)-{\g\over 2}\th_i(r)\si^2 \Big(\sum_{j=1}^N \pi_j(r)\Big)^2\nn\\
&\q+\a^{-1}\(c_i(s)^\a\th_i^\e(r)^{1-{\a\over1-\g}}-\rho_i\th_i^\e(r)\)\bigg]dr\bigg].
\label{proof-VT-3}
\end{align}
Then, by the standard estimates of SDEs, we have
$$
\dbE_t\[\sup_{r\in[t,t+\e]}|X^\e(r)-\bar X(r)|^2\]\les K\e.$$
Noting that $\th_i^\e(t+\e)=\th_i(t+\e)$ and the derivatives $\dot\th_i^\e$ and $\dot\th_i$ are bounded,
we get
$$
\sup_{r\in[t,t+\e]}|\th_i^\e(r)-\th_i(r)|^2\les K\e,
$$
where $K>0$ is a constant which depends on the fixed $\pi$ and $c$, but is independent of $\e$.
Thus, we have
\begin{align}
Y^\e_i(t)&
%
%
%
=\dbE_t\bigg[{\th_i(t)\bar X(t)^{1-\g}\over 1-\g}+\int_t^{t+\e}\bar X(r)^{1-\g}\bigg[ {\dot{\th}_i(r) \over 1-\g}\nn\\
&\q+\th_i(r)\Big(\nu +(\m-\nu)\sum_{j=1}^N \pi_j(r)-\sum_{j=1}^N c_j(r)\Big)-{\g\over 2}\th_i(r)\si^2 \Big(\sum_{j=1}^N \pi_j(r)\Big)^2\nn\\
&\q+\a^{-1}\(c_i(s)^\a\th_i(r)^{1-{\a\over1-\g}}-\rho_i\th_i(r)\)\bigg]dr\bigg]+o(\e).\nn
\end{align}
It follows that
\begin{align}
&J^\l (t,\bar X(t);\bar\pi^\e, \bar c^\e)={1\over N}\sum_{i=1}^N Y^\e_i(t)\nn\\
&\q=\dbE_t\bigg[{\bar X(t)^{1-\g}\over N(1-\g)}\sum_{i=1}^N\th_i(t)+\int_t^{t+\e}\bar X(r)^{1-\g}\cH^\l(r;\pi,c)dr\bigg]+o(\e),\nn
\end{align}
where
\begin{align}
\cH^\l(r;\pi,c)&= {\sum_{i=1}^N\dot{\th}_i(r) \over N(1-\g)}+{1\over N}\sum_{i=1}^N\th_i(r)\bigg[\nu +(\m-\nu)\sum_{j=1}^N \pi_j(r)-\sum_{j=1}^N c_j(r)\nn\\
&\q-{\g\si^2\over 2} \(\sum_{j=1}^N \pi_j(r)\)^2\bigg]+{1\over N\a}\sum_{i=1}^N\(c_i(s)^\a\th_i(r)^{1-{\a\over1-\g}}-\rho_i\th_i(r)\).\nn
\end{align}
Note that
\begin{align}
&J^\l (t,\bar X(t);\bar\pi, \bar c)=\dbE_t\bigg[{\bar X(t)^{1-\g}\over N(1-\g)}\sum_{i=1}^N\th_i(t)+\int_t^{t+\e}\bar X(r)^{1-\g}\cH^\l(r;\bar\pi,\bar c)dr\bigg],\nn
\end{align}
and
$$
\cH^\l(r;\bar\pi,\bar c)=\max_{(\pi,c)}\cH^\l(r;\pi, c)=0.
$$
Then it is easilly checked that \rf{def-equilibrium1} holds, which completes the proof.
\subsection*{Appendix C: Proof of \autoref{thm:well-posedness}}
Similar to the proof of \autoref{prop:well-RU}, it suffices to show that if $\th_i,i=1,...,N$ is a postive solution of \rf{M-HJB}
on $[t_0,T]$, then
$$
{\d}\les\th_i(s)\les \kappa,\q s\in[t_0,T];\q i=1,...,N,
$$
for some positive constants $\d,\kappa>0$ independent of $t_0$.

\ms

\textbf{Step 1. Monotonicity of $\th_i$ in $\rho_i$.}
Without loss of generality, we let $\rho_1\ges\rho_i$ for $i=2,...,N$.
Denote $\D\th_i=\th_1-\th_i$.
Then
$$
\left\{
\begin{aligned}
&\D\dot{\th}_i(t)+(1-\g)\D\th_i(t)\bigg[\nu+{(\mu-\nu)^2\over 2\g\si^2}-\sum_{i=1}^N\bigg({\sum_{j=1}^N\th_j(t)
\over \th_i(t)^{1-\g-\a\over 1-\g}}\bigg)^{1\over\a-1}\bigg]\\
&\qq-(1-\g)\rho_1\a^{-1}\D\th_i(t)+(1-\g)(\rho_i-\rho_1)\a^{-1}\th_i(t)\\
&\qq+{1-\g-\a\over\a(1-\a)}\bigg({1\over N}\sum_{j=1}^N\th_j(t)
\bigg)^{\a\over\a-1}\int_0^1\big[l\th_1(t)+(1-l)\th_i(t)\big]^{-\a\g\over(1-\a)(1-\g)}dl\D\th_i(t)=0,\\
&\D\th_i(T)=0.
\end{aligned}\right.
$$
Clearly, the above is a linear ODE with the unknown $\D\th_i$ and the nonhomogeneous  term
$(1-\g)(\rho_i-\rho_1)\a^{-1}\th_i$. Using the fact
$$(1-\g)(\rho_i-\rho_1)\a^{-1}\th_i(t)\les 0,\qq t\in[t_0,T],$$
we have $\D\th_i\les 0$, which implies that
\bel{Proof-WP1}
\th_1\les\th_i,\qq i=2,...,N.
\ee

\ms

\textbf{Step 2. Low bound estimate of $\th_i$.}
We now show that
$$\th_i\ges \d;\qq i=2,...,N,$$
for some $\d>0$.
From \rf{Proof-WP1}, we only need to prove $\th_1\ges \d$.

\ms

We first prove the result under the assumption \rf{Assum-HJB}.
Denote
\begin{align*}
F(t)&=-(1-\g)\th_1(t)\sum_{i=1}^N\bigg({\sum_{j=1}^N\th_j(t)
\over \th_i(t)^{1-\g-\a\over 1-\g}}\bigg)^{1\over\a-1}\\
&\q+\a^{-1}(1-\g)\th_1(t)^{1-\g-\a\over1-\g}\bigg({\sum_{j=1}^N\th_j(t)
\over \th_1(t)^{1-\g-\a\over 1-\g}}\bigg)^{\a\over\a-1}.
\end{align*}
Note that
\begin{align*}
F(T)=(1-\g)(\a^{-1}-1)>0,
\end{align*}
and
$$
\th_1(t)= 1+\int_t^T \bigg\{(1-\g)\[\nu+{(\mu-\nu)^2\over 2\g\si^2}-\rho_1\a^{-1}\]\th_1(s) +F(s)\bigg\}ds.
$$
%
Then there exists an $S\in[t_{0},T)$ such that $\th_1(t)\ges 1$ on $[S,T]$.
Let
\bel{Proof-WP2}
\t\deq \max\big\{t\in[t_0,T]\,|\, \th_1(t)<1\big\}.
\ee
If $\t=t_0$, then the desired result is proved with $\d=1$.
Thus, we assume that $\t>t_0$. By the continuity of $\th_1$, we get $\th_1(\t)=1$.
Note that
$$
\th_i(t)\ges\th_1(t)\ges 1,\qq t\in [\t,T],\qq\hbox{and}\qq {1-\g-\a\over (1-\g)(1-\a)}<1,
$$
we have
\begin{align*}
F(\t)
&=\bigg[-(1-\g)\sum_{i=1}^N\Big(\th_i(t)^{1-\g-\a\over 1-\g}\Big)^{1\over 1-\a}\bigg(\sum_{j=1}^N\th_j(t)\bigg)^{1\over \a-1}\\
&\qq+\a^{-1}(1-\g)\sum_{i=1}^N\th_i(t)\bigg(\sum_{j=1}^N\th_j(t)\bigg)^{1\over \a-1}\bigg]\bigg|_{t=\t}\\
&\ges \bigg[(\a^{-1}-1)(1-\g)\sum_{i=1}^N\th_i(t)\bigg(\sum_{j=1}^N\th_j(t)\bigg)^{1\over \a-1}\bigg]\bigg|_{t=\t}\\
&>0.
\end{align*}
Thus for some $\e>0$, we have
\begin{align*}
\th_1(t)&= \th_1(\t)+\int_t^\t \[(1-\g)\Big(\nu+{(\mu-\nu)^2\over 2\g\si^2}-\rho_1\a^{-1}\Big)\th_1(s) +F(s)\]ds\\
&= 1+\int_t^\t \[(1-\g)\Big(\nu+{(\mu-\nu)^2\over 2\g\si^2}-\rho_1\a^{-1}\Big)\th_1(s) +F(s)\]ds\\
&\ges 1,\qq t\in[\t-\e,\t],
\end{align*}
which contradicts \rf{Proof-WP2}. This completes the proof under the assumption \rf{Assum-HJB}.

\ms

We next prove the result under the assumption \rf{Assum-HJB'}.
Note that
\begin{align*}
F(t)
&=-(1-\g)\th_1(t)\sum_{i=1}^N\bigg({\sum_{j=1}^N\th_j(t)
\over \th_i(t)^{1-\g-\a\over 1-\g}}\bigg)^{1\over\a-1}\\
&\q+\a^{-1}(1-\g)\th_1(t)^{1-\g-\a\over1-\g}\bigg({\sum_{j=1}^N\th_j(t)
\over \th_1(t)^{1-\g-\a\over 1-\g}}\bigg)^{\a\over\a-1}\\
&=-(1-\g)\th_1(t)\sum_{i=1}^N\bigg({\sum_{j=1}^N\th_j(t)
\over \th_i(t)^{1-\g-\a\over 1-\g}}\bigg)^{1\over\a-1}\\
&\q+\a^{-1}(1-\g)\th_1(t)^{1-\g-\a\over1-\g}{\sum_{i=1}^N\th_i(t)\left(\sum_{j=1}^N\th_j(t)\right)^{1\over \a-1}
\over \th_1(t)^{1-\g-\a\over 1-\g}\left(\th_1(t)^{1-\g-\a\over 1-\g}\right)^{1\over \a-1}}.
\end{align*}
Using the fact $\th_i\ges\th_1$, we have
\begin{align*}
F(t)
&\ges -(1-\g)\th_1(t)\sum_{i=1}^N\bigg({\sum_{j=1}^N\th_j(t)
\over \th_i(t)^{1-\g-\a\over 1-\g}}\bigg)^{1\over\a-1}\\
&\q+\a^{-1}(1-\g)\th_1(t){\sum_{i=1}^N\th_1(t)^{1-\g-\a\over 1-\g}\left(\sum_{j=1}^N\th_j(t)\right)^{1\over \a-1}
\over \th_1(t)^{1-\g-\a\over 1-\g}\left(\th_1(t)^{1-\g-\a\over 1-\g}\right)^{1\over \a-1}}\\
&= -(1-\g)\th_1(t)\sum_{i=1}^N\bigg({\sum_{j=1}^N\th_j(t)
\over \th_i(t)^{1-\g-\a\over 1-\g}}\bigg)^{1\over\a-1}\\
&\q+\a^{-1}(1-\g)\th_1(t)N{\left(\sum_{j=1}^N\th_j(t)\right)^{1\over \a-1}
\over \left(\th_1(t)^{1-\g-\a\over 1-\g}\right)^{1\over \a-1}}.
\end{align*}
Under  \rf{Assum-HJB'}, we have $1-\g-\a\les 0$. Then,
\begin{align*}
F(t)&\ges -(1-\g)\th_1(t)\sum_{i=1}^N\bigg({\sum_{j=1}^N\th_j(t)
\over \th_1(t)^{1-\g-\a\over 1-\g}}\bigg)^{1\over\a-1}\\
&\q+\a^{-1}(1-\g)\th_1(t)N{\left(\sum_{j=1}^N\th_j(t)\right)^{1\over \a-1}
\over \left(\th_1(t)^{1-\g-\a\over 1-\g}\right)^{1\over \a-1}}\\
&= (\a^{-1}-1)(1-\g)\th_1(t)N{\left(\sum_{j=1}^N\th_j(t)\right)^{1\over \a-1}
\over \left(\th_1(t)^{1-\g-\a\over 1-\g}\right)^{1\over \a-1}}\\
&\ges 0.
\end{align*}
Thus,
\begin{align*}
\th_1(t)&\ges e^{\int_t^T(1-\g)\big[\nu+{(\mu-\nu)^2\over 2\g\si^2}-\rho_1\a^{-1}\big]ds}\\
&\ges  e^{-T|1-\g|\big|\nu+{(\mu-\nu)^2\over 2\g\si^2}-\rho_1\a^{-1}\big|}=:\d>0,
\end{align*}
which completes the proof under the assumption \rf{Assum-HJB'}.

\ms

\textbf{Step 2. Upper bound estimate of $\th_i$.}
We first prove the result under the assumption \rf{Assum-HJB}.
Note that under the assumption $\rho_1\les \a\nu+{(\mu-\nu)^2\over 2\g\si}$, we have $\th_i\ges 1$.
It follows that
\begin{align*}
\th_i(t)^{1-\g-\a\over1-\g}\bigg({\sum_{j=1}^N\th_j(t)
\over \th_i(t)^{1-\g-\a\over 1-\g}}\bigg)^{\a\over\a-1}
&=\th_i(t)^{1-\g-\a\over1-\g}\th_i(t)^{\a(1-\g-\a)\over(1-\g)(1-\a)}\bigg(\sum_{j=1}^N\th_j(t)
\bigg)^{\a\over\a-1}\\
&=\th_i(t)^{(1-\g-\a)\over(1-\g)(1-\a)}\bigg(\sum_{j=1}^N\th_j(t)\bigg)^{\a\over\a-1}\\
&\les N^{\a\over\a-1}\th_i(t).
\end{align*}
Thus,
\begin{align*}
&-(1-\g)\th_i(t)\bigg(\sum_{j=1}^N\th_j(t)\bigg)^{1\over\a-1}\sum_{k=1}^N\th_k(t)^{1-\g-\a\over (1-\a)(1-\g)}
+\a^{-1}(1-\g)\th_i(t)^{1-\g-\a\over1-\g}\bigg({\sum_{j=1}^N\th_j(t)\over \th_i(t)^{1-\g-\a\over 1-\g}}\bigg)^{\a\over\a-1}\\
&\q\les \a^{-1}(1-\g) N^{\a\over\a-1}\th_i(t).
\end{align*}
Then from \rf{M-HJB}, we have
\begin{align*}
\th_i(t)&\les e^{(1-\g)\big[\nu+{(\mu-\nu)^2\over 2\g\si^2}-\rho_i\a^{-1}+\a^{-1} N^{\a\over\a-1}\big](T-t)}\\
&\les e^{\big[\nu+{(\mu-\nu)^2\over 2\g\si^2}-\rho_i\a^{-1}+\a^{-1} N^{\a\over\a-1}\big]T}=:\k.
\end{align*}

\ms

We next prove the result under the assumption \rf{Assum-HJB'}.
Note that in this case, ${1-\g-\a\over1-\g}\les 0$, and thus
\begin{align}
&\th_i(t)^{1-\g-\a\over1-\g} \les \d^{1-\g-\a\over1-\g},\qq
\bigg({1\over N}\sum_{i=1}^N\th_i(t)\bigg)^{1\over \a-1}\les \d^{1\over \a-1},\nn\\
& \bigg({1\over N}\sum_{i=1}^N\th_i(t)^{1-\g-\a\over 1-\g}\bigg)^{1\over 1-\a}\les\d^{1-\g-\a\over (1-\a)(1-\g)}.\nn
\end{align}
From the above, we see that all the singular terms in \rf{M-HJB} are bounded.  
Then one can easily find a constant $\kappa>0$, independent of $t_0$, such that
$$
\th_i(t)\les\kappa.
$$

\section*{Acknowledgement}
Hanxiao Wang is supported  by NSFC Grant 12201424,
Guangdong Basic and Applied Basic Research Foundation 2023A1515012104,
and Research Team Cultivation Program of Shenzhen University (2023QNT011).
Chao Zhou is supported by  NSFC Grant 11871364
and Singapore MOE  AcRF Grants A-800453-00-00, R-146-000-271-112, R-146-000-284-114.
The authors  would like to thank Prof. Jiongmin Yong (of UCF) for some suggestive comments
and Dr. Luchan Zhang (of SZU) for some help on numerics.


\begin{thebibliography}{99}

\bibitem{Anderson2005} E. W. Anderson,
\it The dynamics of risk-sensitive allocations,
\rm Journal of Economic theory, {\bf125}  (2005), 93--150.

\bibitem{Aziz-Brandt-Harrenstein2013} H. Aziz, F. Brandt, and P. Harrenstein,
\it Pareto optimality in coalition formation,
\rm Games and Economic Behavior, {\bf82} (2013), 562--581.

\bibitem{Basak-Chabakauri2010} S.~Basak and G.~Chabakauri,
\it Dynamic mean-variance asset allocation,
\rm The Review of Financial Studies, {\bf23} (2010), 2970--3016.


\bibitem{Bjork-Khapko-Murgoci2017} T.~Bj\"{o}rk, M.~Khapko, and A.~Murgoci,
\it On time-inconsistent stochastic control in continuous time,
\rm Finance and Stochastics, {\bf 21} (2017), 331--360.

\bibitem{Bjork-Murgoci-Zhou2014}  T.~Bj\"{o}rk, A.~Murgoci, and X.~Y.~Zhou,
\it Mean-variance portfolio optimization with state-dependent risk aversion,
\rm Mathematical Finance, {\bf 24} (2014),  1--24.

\bibitem{Borovicka2020} J. Borovi\v{c}ka,
\it Survival and long-run dynamics with heterogeneous beliefs under recursive preferences,
\rm Journal of Political Economy, {\bf 128} (2020), 206--251.

\bibitem{Cao2018} D. Cao and I. Werning,
\it Saving and dissaving with hyperbolic discounting
\rm Econometrica, {\bf 86} (2018), 805--857.


\bibitem{Cvitanic-Zhang2012} J. Cvitani\'{c} and J. Zhang,
\it Contract theory in continuous-time models,
\rm Springer Science and Business Media, 2012.

\bibitem{Dana1990} R. A. Dana and C. LeVan,
\it Structure of Pareto optima in an infinite-horizon economy
where agents have recursive preferences,
\rm Journal of Optimaztion Theory and Applications, {\bf 64} (1990), 269--292.

\bibitem{Dai2021} M. Dai, H. Jin, S. Kou, and Y. Xu,
\it A dynamic mean-variance analysis for log returns,
\rm Management Science, {\rm 67} (2021), 1093--1108.

\bibitem{Dokuchaev1999} M.~Dokuchaev and X.~Y.~Zhou,
\it Stochastic controls with terminal contingent conditions,
\rm Journal of Mathematical Analysis and Applications, {\bf238} (1999), 143--165.

\bibitem{Duffie-Epstein1992} D.~Duffie and L.~G.~Epstein,
\it Stochastic differential utility,
\rm Econometrica, {\bf 60} (1992), 353--394.

\bibitem{Duffie-Epstein1992-1} D.~Duffie and L.~G.~Epstein,
\it Asset pricing with stochastic differential utility,
\rm The Review of Financial Studies, {\bf 5} (1992), 411--436.

\bibitem{Duffie-Geoffard-Skiadas-1994} D.~Duffie, P.~Y.~Geoffard,  and C.~Skiadas,
\it Efficient and equilibrium allocations with stochastic differential utility,
\rm Journal of Mathematical Economics, {\bf23} (1994), 133--146.

\bibitem{Duffie-Lions1992} D. Duffie and P. L. Lions,
\it PDE solutions of stochastic differential utility,
\rm Journal of Mathematical Economics, {\bf 21} (1992), 577--606.

\bibitem{Dumas2000} B. Dumas, R. Uppal, and T. Wang,
\it Efficient Intertemporal Allocations with Recursive Utility,
\rm Journal of Economic Theory, {\bf93} (2000),  240--259.



\bibitem{Ekeland2010} I.~Ekeland and A.~Lazrak,
\it The golden rule when preferences are time inconsistent,
\sl  Mathematics and Financial Economics,
\rm {\bf 4} (2010), 29--55.

\bibitem{Ekeland2012} I. Ekeland, O. Mbodji, and  T.~A.~Pirvu,
\it Time-consistent portfolio management,
\rm SIAM Journal on Financial Mathematics, {\bf3} (2012),  1--32.

\bibitem{Ekeland2008} I.~Ekeland and T.~A.~Pirvu,
\it Investment and consumption without commitment,
\rm Mathematics and Financial Economics, {\bf 2} (2008), 57--86.



\bibitem{El Karoui-Peng-Quenez1997} N.~EI~Karoui, S.~Peng, and M.~C.~Quenez,
\it Backward stochastic differential equations in finance,
\rm Mathematical Finance, {\bf7}  (1997), 1--71.



\bibitem{Epstein-Zin1989} L. Epstein and S. Zin,
\it Substitution, Risk Aversion and the Temporal Behavior of
Consumption and Asset Returns: A Theoretical Framework,
\rm Econometrica, {\bf57} (1989), 937--969.

\bibitem{Garlappi2017} L. Garlappi, R. Giammarino, and A. Lazrak,
\it Ambiguity and the corporation: Group disagreement and underinvestment,
\rm Journal of Financial Economics, {\bf125} (2017), 417--433.

\bibitem{Garleanu2015} N. G\^{a}rleanu and S. Panageas,
\it Young, old, conservative, and bold: The implications of heterogeneity and finite lives for asset pricing,
\rm Journal of Political Economy, {\bf123} (2015),  670--685.

\bibitem{Harris} C. Harris and D. Laibson,
\it Dynamic choices of hyperbolic consumers,
\rm Econometrica, {\bf69}  (2001), 935--957.


\bibitem{Hong-Wang2021} H. Hong, N. Wang, and J. Yang,
\it Welfare consequences of sustainable finance,
\rm National Bureau of Economic Research, 2021

\bibitem{Hu-Ji-Xue2018} M.~Hu, S.~Ji, and X.~Xue,
\it A global stochastic maximum principle for fully coupled forward-backward stochastic systems,
\rm SIAM Journal on Control and Optimization, {\bf56} (2018), 4309--4335.



\bibitem{Hu-Jin-Zhou2012} Y.~Hu, H.~Jin, and X.~Y.~Zhou,
\it Time-inconsistent stochastic linear--quadratic control,
\rm SIAM Journal on Control and Optimization, {\bf 50} (2012), 1548--1572.

\bibitem{Jackson-Yariv2015} M. O. Jackson and L. Yariv,
\it Collective dynamic choice: the necessity of time inconsistency,
\rm American Economic Journal: Microeconomics, {\bf7} (2015), 150--178.

\bibitem{Kan1994} R. Kan,
\it Structure of Pareto optima when agents have stochastic recursive preferences,
\rm Journal of Economic Theory, {\bf64} (1995), 626--631.

\bibitem{Krusell} P. Krusell and A. A. Smith Jr,
\it Consumption and saving decisions with quasi-geometric discounting,
\rm Econometrica, {\bf71}  (2003), 366--375.

\bibitem{Laibson1997} D. Laibson,
\it Golden eggs and hyperbolic discounting,
\rm Quarterly Journal of Economics, {\bf112} (1997), 443--477.

\bibitem{Lim-Zhou2001} A.~E.~B.~Lim and X.~Y.~Zhou,
\it Linear-quadratic control of backward stochastic differential equations,
\rm SIAM Journal on Control and Optimization, {\bf40} (2001), 450--474.

\bibitem{Ma1993} C. Ma,
\it Market equilibrium with heterogeneous recursive-utility-maximizing agents,
\rm Economic Theory, {\bf3}  (1993), 243--266.



\bibitem{Morellec} E. Morellec and N. Schš¹rhoff,
\it Corporate investment and financing under asymmetric information,
\rm Journal of Financial Economics, {\bf 99} (2011),  262--288.

\bibitem{Peng1993} S.~Peng,
\it Backward stochastic differential equations and applications to optimal control,
\rm Applied Mathematics and Optimization, {\bf 27} (1993), 125--144.

\bibitem{Peng1997} S.~Peng,
\it Backward stochastic differential equations and stochastic optimizations,
\rm Topics in Stochastic Analysis, J.~Yan, S.~Peng, S.~Fang, and L.~Wu, eds., Science Press, Beijing, 1997 (in Chinese).

\bibitem{Pollak1968} R.~A.~Pollak,
\it Consistent planning,
\rm  The Review of Economic Studies, {\bf 35} (1968), 185--199.

\bibitem{Pindyck2013} R. S. Pindyck and N. Wang,
\it The economic and policy consequences of catastrophes,
\rm American Economic Journal: Economic Policy, {\bf 5} (2013),  306--339.



\bibitem{Sun-Wang-Wen2023} J.~Sun, H.~Wang, and J.~Wen,
\it Zero-sum Stackelberg stochastic linear-quadratic differential games,
\rm SIAM Journal on Control and Optimization, {\bf61} (2023),  250--282.






\bibitem{Strotz1955} R.~H.~Strotz,
\it Myopia and inconsistency in dynamic utility maximization,
\rm The Review of Economic Studies, {\bf 23} (1955), 165--180.

\bibitem{von Neumann1944} J. von Neumann and O. Morgenstern,
\it Theory of games and economic behavior,
\rm New York: John Wiley \& Sons, 1944.

\bibitem{Wang-Wu-Xiong2013} G.~Wang, Z.~Wu, and J.~Xiong,
\it Maximum principles for forward-backward stochastic control systems with correlated state and observation noises,
\rm SIAM Journal on Control and Optimization, {\bf51} (2013), 491--524.

\bibitem{Wang-Yong-Zhou2022} H. Wang, J. Yong, and C. Zhou,
\it Optimal controls for forward-backward stochastic differential equations: Time-inconsistency and time-consistent solutions,
\rm preprint arXiv:2209.08994, 2022.

\bibitem{Yong2010} J.~Yong,
\it Optimality variational principle for controlled forward-backward stochastic differential equations with mixed initial-terminal conditions,
\rm SIAM Journal on Control and Optimization, {\bf48} (2010), 4119--4156.

\bibitem{Yong2012} J.~Yong,
\it Time-inconsistent optimal control problems and the equilibrium HJB equation,
\rm Mathematical Control and Related Fields, {\bf2} (2012), 271--329.

\bibitem{Yong2014} J.~Yong,
\it Time-inconsistent optimal control problems,
\rm Proceedings of 2014 ICM, Section 16. Control Theory and Optimization, (2014), 947--969.

\bibitem{Yong2014book} J. Yong,
\it Differential games: a concise introduction,
\rm World scientific, 2014.

\end{thebibliography}
\end{document}